\documentclass{article}

\usepackage{mytemplate}

\newtheorem{thm}{Theorem}
\newtheorem{cor}[thm]{Corollary}
\newtheorem{prop}[thm]{Proposition}

\newcommand{\p}{\mathcal{P}}
\newcommand{\h}{{\mathrm{ht}\,}}
\newcommand{\hl}{{\mathrm{ht}_L}}
\newcommand{\hk}{{\mathrm{ht}_K}}
\newcommand{\fqb}{{\overline{\F}_q}}
\newcommand{\fqn}{{\F_{q^\nu}}}
\newcommand{\fn}{{\F_{q^n}}}
\newcommand{\kb}{\overline{K}}
\newcommand{\pdiv}{{\mathrm{PDiv}}}

\begin{document}

\title{Greatest Prime Divisors of Polynomial Values over Function Fields}
\author{Alexei Entin}
\maketitle

\begin{abstract} For a function field $K$ and fixed polynomial $F\in K[x]$ and varying $f\in F$ (under certain restrictions) we give a lower bound for the degree of the greatest prime divisor of $F(f)$ in terms of the height of $f$, establishing a strong result for the function field analogue of a classical problem in number theory.\end{abstract}

\section{Introduction} Let $F\in\Z[x]$ ($\Z$ denotes the ring of integers) be a fixed polynomial with integer coefficients, degree $\deg F\ge 2$ and with distinct roots (over
the complex numbers). For $N\in\Z, N>1$ denote by $\p(N)$ the largest prime factor of
$N$. The problem of giving a lower bound for $\p(F(n))$ in terms of $n$ as $n\to\ity$ has been much studied.
P\'olya \cite{polya} proved that $\p(F(n))\to\ity$ as $n\to\ity$ for the case $\deg F=2$ and the general case can be deduced from Siegel's theorem on the finiteness of integer points on curves with positive genus.
Keates \cite{keates} proved a bound of the form \beq\label{weak}\p(F(n))\gg_F \log\log n\eeq (the implicit constant depending on $F$) for $\deg F=2,3$, after some special cases had been proved
by Mahler \cite{mahler}, Nagell \cite{nagell} and Schinzel \cite{schinzel}. Kotov \cite{kotov}, building on the work of Sprindzhuk \cite{sprindzhuk}, extended this result to $F$ of any degree $\ge 2$.
It is conjectured that in fact \beq\label{nfconj}\p(F(n))>(\deg F-1-\eps)\log n\eeq for any fixed $F,\eps>0$ and $n$ sufficiently large.
This would follow from the conditional results of Granville \cite{granville} and Langevin \cite{langevin1,langevin2}, which assume the ABC-conjecture.

We are concerned with the function field analogue of this problem. Let $p$ be a prime, $q$ its power and $\F_q$ the field with $q$ elements.
Let $K$ be the function field of the curve $C_K$ defined over $\F_q$. By a curve we will always mean a smooth projective algebraic curve.
For a function $f\in K^\times$ we denote by $(f)$ its divisor which can be decomposed into its zero and polar components $(f)=(f)_0-(f)_\ity$.
The \emph{height} of $f$ is defined to be $\h f=\deg(f)_0=\deg(f)_\ity$. For a divisor $D$ on $K$ we denote by $\sup D$ its support
(the set of prime divisors appearing in $D$ with nonzero coefficient) and define $$\del(f)=\max_{P\in\sup(f)}\deg P$$ (for $f$ not constant).

Now fix $p,q,K,0\neq F\in K[x]$. We are concerned with a lower bound for $\del(F(f))$ in terms of $\h f$ as $\h f\to\ity$. For this problem we may
assume without loss of generality that $F$ has no repeated irreducible factors in $K[x]$ (i.e. squarefree), otherwise just replace it with the product of its irreducible factors. In the case of function fields it can happen that $\del(F(f))$ stays bounded while $\h f\to\ity$. For example
if $F\in\F_q[x]$ has constant coefficients, $t\in K$ and $f=t^{q^k}$, then $F(f)=F(t)^{q^k}$, so $\del(F(f))=\del(F(t))$ while $\h f\to\ity$ as $k\to\ity$.
Under certain restrictions on $F$ such pathologies do not occur and we will obtain a bound analogous to \rf{nfconj}. 

We denote by $\fqb\ss\kb$ the algebraic closures of $\F_q,K$ respectively. A polynomial $F\in K[x]$ is called separable if it has distinct roots over $\kb$. Our main result is the following
\begin{thm}\label{main} Let $q,K,0\neq F\in K[x]$ be fixed with $F$ squarefree (in $K[x]$). Assume that $F$ is either non-separable or has (at least)
three distinct roots $a_1,a_2,a_3$ in $\kb$ s.t.
$$\frac{a_1-a_2}{a_1-a_3}\not\in\fqb.$$
Then there exists a constant $\lam$ depending on $F$ s.t.
$$\del(F(f))>\log_q\h f-\lam$$ for all $f\in K$ (for which $F(f)$ is not constant).\end{thm}

We call a separable squarefree polynomial $F\in K[x]$ \emph{exceptional} if it fails the condition of Theorem \ref{main}. This is the function field analogue
of the exceptional polynomials in $\Z[x]$ as defined in \cite{kotov}, to which the main method of \cite{kotov} is not applicable but can be treated
by other means (to obtain the bound (\ref{weak})). 
Our notion of exceptional polynomial should not be confused with the notion of exceptional polynomials over $\F_q$ as defined in \cite{davenport}. 
For exceptional polynomials we will obtain the following result:

\begin{thm}\label{comp} Let $F\in K[x]$ be a fixed separable polynomial.
\begin{enumerate}\item The polynomial $F$ is exceptional if and only if there exist $s,t\in K,n\ge 0$ s.t. $F$ divides the polynomial
\beq\label{exmod}x^{q^n}-sx+t\eeq and the latter polynomial is nonzero.
\item If $F$ is exceptional then there is a sequence $f_k\in K$ s.t. $\h f_k\to\ity$ as $k\to\ity$ but $\del(F(f_k))$ stays bounded.
\item Assume that $F$ is exceptional and divides the nonzero polynomial \rf{exmod} for some $s,t\in K,n\ge 0$. Let $\eps>0$ be fixed. Then for $f\in K$ such that
$sf-t$ is not a $p$-th power in $K$ and $\h f$ is sufficiently large (i.e. larger than some constant depending only on $F,\eps$) we have
\beq\label{exineq}\del(F(f))> \log_q\h f+\log_q(\deg F-1)+\log_q(1-1/q)-\eps.\eeq\end{enumerate}
\end{thm}

\begin{cor} Let $F\in\F_q[x]$ be a fixed squarefree polynomial with constant coefficients. Then \rf{exineq} holds (for any fixed $\eps>0$) whenever $f\in K$ is not a $p$-th power in $K$ and $\h f$ is sufficiently large.\end{cor} 

\begin{proof} A squarefree polynomial with constant coefficients always divides a polynomial of the form $x^{q^n}-x$, so we can apply Theorem \ref{comp} with
$s=1,t=0$.\end{proof}

\section{Preliminaries}

For the proof of our results we will need the following proposition, which is an extension of the ABC-theorem for function fields.

\begin{prop}\label{abcd} Let $K$ be the function field of the curve $C_K$ over $\F_q$ with genus $g_K$. Let $u\in K$ be a function which is not a $p$-th power
in $K$ and $b_1,...,b_m\in\F_q$.
Then
$$\sum_{P\in\cup\sup(u-b_i)}\deg P\ge (m-1)\h u-(2g_K-2).$$
\end{prop}

\begin{proof} Consider the extension $\F_q(u)\ss K$. This is a separable geometric extension of function fields (because $u$ is not
a $p$-th power in $K$) of degree $\h u$, so we may apply the Riemann-Hurwitz formula to obtain
$$2g_K-2\ge -2\h u+\sum_P (e_P-1)\deg P,$$ where $e_P$ is the ramification index of the prime $P$ of $K$ in this extension (equality
is obtained if all the $e_P$ are coprime to $p$, but we do not assume this).
Restricting to the primes $P\in\cup_{i=1}^m\sup(u-b_i)$, which are exactly the primes lying over the primes $\F_q(u)$ corresponding to the points
$b_1,...,b_m,\ity$ on $\mathbf{P}^1$ (considering $\F_q(u)$ as the function field of $\mathbf{P}^1$) and using
$$\sum_{P\in\cup\sup(u-b_i)}e_P\deg P=[K:\F_q(u)]\cdot\#\{b_1,...,b_m,\ity\}=(m+1)\h u$$
we obtain 
$$2g_k-2+2\h u\ge \sum_{P\in\cup\sup(u-b_i)}(e_P-1)\deg P=(m+1)\h u-\sum_{P\in\cup\sup(u-b_i)}\deg P.$$ Therefore $$\sum_{P\in\cup\sup(u-b_i)}\deg P\ge (m-1)\h u-(2g_K-2),$$ as required.\end{proof}

Taking $m=2,b_1=0,b_2=1$ in the last proposition we obtain the ABC-theorem for function fields in the following form (see also \cite[Theorem 7.17]{rosen}):

\begin{prop}\label{abc} Let $K$ be the function field of the curve $C_K$ over $\F_q$ with genus $g_K$. Let $u\in K$ be a function which is not a $p$-th power
in $K$. Then
$$\sum_{P\in\sup(u)\cup\sup(u-1)}\deg P\ge \h u-(2g_K-2).$$\end{prop}

\section{Proof of Theorem \ref{main}}

Let $K$ be the function field of the curve $C_K$ defined over $\F_q$ and let $F\in K[x]$ be a squarefree polynomial. We assume without loss of generality
that $F$ is monic (if $c$ is the leading coefficient of $F$ then $\del(F(f))=\del(F(f)/c)$ whenever $\del(F(f))>\del(c)$).

\begin{prop}\label{main1} Assume that there exist three distinct roots $a_1,a_2,a_3\in\kb$ of $F$ s.t. $$\tau=\frac{a_1-a_2}{a_1-a_3}\not\in\fqb.$$ Then the
assertion of Theorem \ref{main} holds for $F$.\end{prop}

\begin{proof}
Let $L$ be the splitting field of $F$ over $K$, $C_L$ its underlying curve with genus $g_L$.
Since $\tau\not\in\fqb$, for some $k$ the element $\tau\in L$ is not a $p^k$-th power in $L$.
Take any $f\in K$.
Denote $$u=\frac{f-a_2}{a_1-a_2},v=\frac{f-a_3}{a_1-a_3}.$$ It is not possible that both $u$ and $v$ are $p^k$-th powers in $L$ because then
so would be $\tau=v/u$ which we assumed is not the case. Assume (by symmetry) that $u$ is not a $p^k$-th power and let $l\le k$ be the largest integer
s.t. $u$ is a $p^l$-th power in $L$. Applying Proposition \ref{abc} to the function $u^{1/p^l}\in L$ (which is not a $p$-th power) we obtain
\beq\label{supu}\sum_{P\in\sup(u)\cup\sup(u-1)}\deg P\ge p^{-l}\h u-(2g_L-2).\eeq
We note that when considering degrees of divisors on $C_L$ we always consider the degree over the field of constants of $C_L$ (which is a finite extension
of $F_q$) and not over $F_q$ itself.
For a function $h\in K^\times$ we will denote by $(h)_K,\hk h$ its divisor and height (respectively) over $K$ and similarly for $L$.

Note that $$u-1=\frac{f-a_1}{a_1-a_2}.$$ Let $a_4,...,a_{\deg_F}\in L$ be the other roots of $F$, so that $F(x)=\prod_{i=1}^{\deg F}(x-a_i)$.
We have \beq\label{ffprod}F(f)=\prod_{i=1}^{\deg F}(f-a_i)=u(u-1)(a_1-a_2)^2\prod_{i=3}^{\deg F}(f-a_i).\eeq
Denote $$M=\max_{1\le i,j\le\deg F\atop{a_i\neq a_j}}\max_{P\in\sup(a_i-a_j)}\deg P.$$ Let $P$ be a prime divisor of $L$ with $\deg P>M$.
Then $P$ is a pole of $f-a_i$ for one $i$ iff it is a pole of each $f-a_j,1\le j\le\deg F$. Also $P\in\sup(u)$ iff $P\in\sup(f-a_2)$ and $P\in\sup(u-1)$
iff $P\in\sup(f-a_1)$. We see that if $P\in\sup(u)$ then it cannot cancel out in the product on the right hand side of \rf{ffprod} and so
$P\in\sup(F(f))_L$. The same holds if $P\in\sup(u-1)$.

We will denote by $O(1)$ quantities which are bounded by a constant depending only on $F$. We will use the notation $P\in\pdiv(L)$ to mean that $P$ is a prime divisor of $L$
and similarly with $K$. We have
$$\sum_{P\in\pdiv(L)\atop{\deg P\le M}}\deg P=O(1)$$ and so using \rf{supu} we obtain
\beq\label{ineq1}\sum_{P\in\sup(F(f))_L}\deg P\ge\sum_{P\in\sup(u)\cup\sup(u-1)}\deg P-O(1)\ge p^{-l}\hl u-O(1).\eeq 
For any prime divisor $Q$ of $K$ we have $$\sum_{P\in C_L\atop{P\,\mathrm{over}\,Q}}\deg P\le\frac{[L:K]}{\nu}\deg Q,$$
where $\F_{q^\nu}$ is the field of constants of $L$ (equality occurs if $Q$ is unramified). Therefore
$$\sum_{Q\in\sup(F(f))_K}\deg Q\ge\frac{\nu p^{-l}}{[L:K]}\hl u-O(1).$$ 
Assuming that $\hl f>\hl a_1$ we see that $\hl f=\hl u+O(1)$ and using $\hl f=\frac{[L:K]}{\nu}\hk f$ we obtain
\beq\label{mainineq}\sum_{Q\in\sup(F(f))_K}\deg Q\ge p^{-l}\hk f-O(1).\eeq

Let $d=\del(F(f))$ be the degree of the largest prime divisor of $K$ appearing in the support of $(F(f))_K$. By the prime number theorem for
function fields (see \cite[Theorem 5.12]{rosen}) for every natural number $e$ we have
$$\sum_{Q\in C_K\atop{\deg q=e}}\deg Q=q^e(1+o(1)),$$ with the $o(1)$ term tending to zero as $e\to\ity$, so
$$\sum_{Q\in C_K\atop{\deg q\le d}}\deg Q\le(1+o(1))\lb\sum_{e=1}^d q^e\rb=(1+o(1))(1-1/q)^{-1}q^d.$$
By \rf{mainineq} we obtain
\beq\label{mainineq1}(1+o(1))(1-1/q)^{-1}q^d\ge p^{-l}\hk f-O(1)\eeq and taking logarithms this becomes
$$\del(F(f))=d\ge \log_q\hk f-O(1)$$ as required.
\end{proof}

\begin{prop}\label{main2} Assume that $F$ is non-separable. Then the assertion of Theorem \ref{main} holds for $F$.\end{prop}

\begin{proof} Since $F$ is squarefree and non-separable it has a non-separable irreducible factor $F_1\in K[x]$. It must be of the form $F_1(x)=G(x^p)$ with $G\in K[x]\sm\F_q[x]$ monic. Of course $G$ is also irreducible
over $K$. Let $L$ be the maximal separable extension of $K$ contained in the splitting field of $G$ over $K$. Over $L$ we have a factorization of the form
\beq\label{g1fact} G(x)=\prod_{i=1}^m\lb x^{p^r}-\al_i\rb,\al_i\in L\eeq
for some $r\ge 0$, with the $\al_i$ distinct.
If all the $\al_i$ are $p$-th powers in $L$ then the coefficients of $G$ are $p$-th powers in $L$ and therefore also in $K$ (because the
extension $K\ss L$ is separable), so $F_1(x)=G(x^p)$ is a $p$-th power of a polynomial in $K[x]$, which is impossible because $F_1$ divides $F$ and
$F$ is squarefree. Therefore we may assume one and therefore all the $\al_i$ (since they are conjugate over $K$) are not $p$-th power in $L$.

Take some $f\in K$ and denote $u=\al_1^{-1}f^{p^{r+1}}$. Since $\al_1$ is not a $p$-th power in $L$, neither is $u$. We have 
$$u-1=\al_1^{-1}\lb f^{p^{r+1}}-\al_1\rb.$$ We may apply Proposition \ref{abc} to $u$ to obtain
$$\sum_{P\in\sup(u)\cup\sup(u-1)}\deg P>\hl u-O(1)=p^{r+1}\hl f-O(1).$$
However $$\sum_{P\in\sup(u)}\deg P\le\sum_{P\in\sup(f)_L}\deg P+\sum_{P\in\sup(\al_1)}\deg P=\hl f+O(1),$$
so \beq\label{supu1}\sum_{P\in\sup(u-1)}\deg P>\lb p^{r+1}-1\rb\hl f-O(1).\eeq
By \rf{g1fact} we have
$$F_1(f)=\prod_{i=1}^m\lb f^{p^{r+1}}-\al_i\rb.$$
As in the proof of Case 1 we see that a prime divisor $P$ of $L$ of sufficiently large degree (depending only on $F$) occuring in $\sup(u-1)_L$
must also occur in $\sup(F_1(f))_L$. Using \rf{supu1} and arguing in the same way as in the proof of Proposition \ref{main2} we obtain
$$\del(F_1(f))>\log_q\hl f-O(1).$$

Denote $H=F/F_1\in K[x]$. Let $Q$ be a prime divisor of $K$. There
exists a constant $N$ depending only on $F$ s.t. if $\deg Q>N$ then $Q$ is a pole of either $F_1(f),H(f)$ iff it is a pole of $f$ (we just take $N$
to be the maximum of the degrees of all the poles of the coefficients of $F_1,H$). For such $Q$ if $Q\in\sup(F_1(f))_K$ then also $Q\in\sup(F(f))_K$
(zeroes and poles cannot cancel out by those of $H(f)$). Therefore $$\del(F(f))>\del(F_1(f))-O(1)>\log_q\hl-O(1),$$
as required.\end{proof}

Now Theorem \ref{main} follows by combining Propositions \ref{main1} and \ref{main2}.

\section{Proof of Theorem \ref{comp}}

Let $F\in K[x]$ be a separable polynomial of degree $m=\deg F$, $a_1,...,a_m\in\kb$ the roots of $F$.

\subsection{Proof of Theorem \ref{comp}(i)}

Let $F$ be exceptional. We want to show that it must divide a nonzero polynomial of the form 
\beq\label{exmod1}x^{q^n}-sx+t,s,t\in K,n\ge 0,\eeq as asserted in Theorem \ref{comp}(i). 
Let $L$ be the splitting field of $F$ over $K$ and $\F_{q^\nu}$ the field of constants in $L$.
Since $F$
is exceptional, for all distinct $1\le i,j,k\le m$ we have $(a_i-a_j)/(a_i-a_k)\in\fqb$. Equivalently, there exist
$\al,\be\in L$ and $b_i\in\fqn,1\le i\le m$ s.t. $a_i=\al b_i+\be$. Consider the polynomial
\beq\label{g}G(x)=\prod_{b\in\fqn}(x-\al b-\be)=x^{q^\nu}-\al^{q^{\nu-1}}x+\al^{q^{n-1}}\be-\be^{q^n}\in L[x]\eeq
(to see that this identity holds just substitute $\al b+\be$ into the RHS to see that it is a root for every $b\in\fqn$).
If $G\in K[x]$ then $F$ divides $G$ which has the form \rf{exmod1}. If $G\not\in K[x]$ then there exists an automorphism $\sig$ of $L$ over $K$ s.t. $G^{\sig}\neq G$ ($G^\sig$ is obtained from $G$ by applying $\sig$ to
each coefficient). The polynomial $F$ divides $G$ and since $F\in K[x]$ it also divides $G^\sig$. Therefore $F$ divides $G-G^\sig$. 
But from \rf{g} we see that $G-G^\sig$ is linear, so $F$ must be linear and already has the form \rf{exmod1}. 
This concludes the proof of one implication of Theorem \ref{comp}(i).

To prove the other implication assume that $F$ divides $G(x)=x^{q^n}-sx+t$ for some $n\ge 1,s,t\in K$ (if $F$ is linear it is obviously exceptional, so we may assume $n\ge 1$). There exist $\al,\be\in\kb$ s.t.
$$\al^{q^n}=s,\al^{q^n-1}\be-\be^{q^n}=t.$$ The roots of $G$ over $\kb$ are precisely $$\al b+\be, b\in\F_{q^n},$$ so the roots of $F$ have the form
$a_i=\al b_i+\be,b_i\in\F_{q^n}$ and $F$ is exceptional. This concludes the proof of Theorem \ref{comp}(i).

\subsection{Proof of Theorem \ref{comp}(ii)}

Suppose $F$ is exceptional and therefore divides the nonzero polynomial $$G(x)=x^{q^n}-sx+t,s,t\in K.$$ Since the assertion of Theorem \ref{comp}(ii) is trivial
for $F$ linear we assume that $n\ge 1$. If $s=0$ then $G$ has only one root over $\kb$ and $F$ cannot be separable unless it is linear. Hence we assume
that $s\neq 0$. Choose some $f_0\in K$ with a pole $P$ of degree $\deg P>\del(s),\del(t)$ and define recursively
$$f_{k+1}=\frac{f_k^{q^n}+t}{s},k\ge 0.$$
The prime divisor $\deg P$ is a pole of multiplicity $q^{kn}$ of $f_{k+1}$, therefore $\h f_{k+1}\to\ity$ as $k\to\ity$.
Now observe that
\begin{multline*}G(f_{k+1})=f_{k+1}^{q^n}-sf_{k+1}+t=\frac{f_k^{q^{2n}}+t^{q^n}}{s^{q^n}}-f_k^{q^n}=\\=\frac{\lb f_k^{q^n}-sf_k+t\rb^{q^n}}{s^{q^n}}=\lb\frac{G(f_k)}{s}\rb^{q^n},\end{multline*}
so $$\del(G(f_{k+1}))\le\max(\del(G(f_k)),s)$$ and therefore $\del(F(f_k))$ stays bounded as $k\to\ity$.

Now write $G(x)=F(x)H(x),H\in K[x]$. Let $P$ be a prime divisor not appearing in the supports of the coefficients of $F$,$H$. Then for any $f\in K^\times$, $P$ is a pole of $F(f)$ iff it is a pole of $f$ and of $H(f)$. We see that sufficiently large (depending only on the coefficients of $F$,$H$) prime divisor cannot
be canceled out when we multiply $F(f)$ by $H(f)$, so if $\del(F(f))$ is sufficiently large we have $\del(G(f))\ge\del(F(f))$,
so $\del(F(f_k))$ is also bounded as $k\to\ity$. This concludes the proof of Theorem \ref{comp}(ii).
 
\subsection{Proof of Theorem \ref{comp}(iii)}

Assume that $F$ is exceptional. As in the proof of Theorem \ref{main} we will also assume without loss of generality that $F$ is monic. If $F$ is linear the assertion of Theorem \ref{comp}(iii) is obvious, so we assume $\deg F=m\ge 2$. Assume that $F$ divides $$G(x)=x^{q^n}-sx+t,s,t\in K,n\ge 1.$$ 
We fix one such $G$ once and for all, so $s,t,n$ are also fixed. We have $s\neq 0$, otherwise $F$ would be linear. It follows that $G$ is separable because
its derivative is $-s\in K^\times$. Let $\al,\be\in\kb$ be such that $$\al^{q^n}=s,\al^{q^n-1}\be-\be^{q^n}=t.$$ Then the roots of $G$ are
$\al b+\be,b\in\F_{q^n}$ and the roots of $F$ have the form $a_i=\al b_i+\be,b_i\in\F_{q^n},1\le i\le m$.
Let $L$ be the splitting field of $G$ over $K\F_{q^n}$ (the composite of the fields $K,\F_{q^n}$). It is a separable extension of $K$ since $G$ is separable.
Since $\deg G\ge 2$ and $s\neq 0$, $G$ has at least two
distinct roots $\al b+\be,\al b'+\be$ from which it follows that $\al,\be\in L$.

Now let $f\in K$ be such that $sf-t$ is not a $p$-th power in $K$. Denote $u=(f-\be)/\al\in L$. We claim that $u$ is not a $p$-th power in $L$. 
Suppose to the contrary that $u$ is a $p$-th power. Then so is $u-b$ for any $b\in\fn$. Now
$$G(f)=\prod_{b\in\F_{q^n}}(f-\al b-\be)=\al^{q^n}\prod_{b\in\F_{q^n}}(u-b),$$ so $G(f)$ is a $p$-th power in $L$ and therefore also in $K$, because $K\ss L$
is a separable extension. But $G(f)=f^{q^n}-sf+t$, so $sf-t$ is also a $p$-th power, a contradiction. Therefore $u$ is not a $p$-th power.

Now we apply Proposition \ref{abcd} to the field $L$, function $u$ and constants $b_1,...,b_m$. We obtain the inequality 
$$\sum_{P\in\cup_{i=1}^m\sup(u-b_i)_L}\deg P>(m-1)\hl u-O(1),$$ where $O(1)$ stands for a quantity bounded by a constant depending only on $F$ and $G$
(the latter was fixed for a given exceptional $F$). Since 
$$F(f)=\prod_{i=1}^m(f-\al b_i-\be)=\al^m\prod_{i=1}^m(u-b_i)$$ we see that a prime divisor $P$ of $L$ with $\deg P>\del(\al)$ and appearing in 
$\cup_{i=1}^m\sup(u-b_i)_L$ must also appear in $\sup(F(f))_L$ (note that $u-b_i$ have the same poles for $1\le i\le m$),
so \beq\label{ineq}\sum_{P\in\sup(F(f))_L}\deg P>(m-1)\hl u-O(1).\eeq

Now denoting $d=\del(F(f))$ and arguing as in the proof of Proposition \ref{main1} (where we deduced \rf{mainineq1} from \rf{ineq1}) we deduce from \rf{ineq} that
$$(1+o(1))(1-1/q)^{-1}q^d>(m-1)\hk f-O(1),$$
($o(1)$ is a quantity tending to 0 as $\hk f\to\ity$ for fixed $F,G$)
from which it follows that
$$d>\log_q\hk f+\log_q(m-1)+\log_q(1-1/q)-o(1),$$ which is exactly the assertion of Theorem \ref{comp}(iii).

{\bf Acknowledgment.} The author would like to thank Ze\'{e}v Rudnick for suggesting the problem studied in the present paper and for his valuable remarks
on previous versions of the paper. The present work is part of the author's Ph.D. studies at Tel-Aviv University under his supervision. The author
would also like to thank Andrzej Schinzel for his useful remarks on this paper.

The research leading to these results has received funding from the European Research Council under the European Union's Seventh Framework Programme (FP7/2007-2013) / ERC grant agreement no 320755.

\end{document}